\documentclass[11pt]{amsart}

\usepackage{amsmath}
\usepackage{fullpage}
\usepackage{xspace}
\usepackage[psamsfonts]{amssymb}
\usepackage[latin1]{inputenc}
\usepackage{graphicx,color}
\usepackage[curve]{xypic} 
\usepackage{hyperref}
\usepackage{graphicx}
\usepackage{amsmath}%
\usepackage{amsthm}%
\usepackage{amscd}
\usepackage{amsfonts}%
\usepackage{amssymb}%
\usepackage{graphicx}
\usepackage{mathrsfs}
\usepackage{tikz}
\usetikzlibrary{matrix,arrows}
\usepackage{tikz-cd}
\usepackage[all]{xy}

\newtheorem{theorem}{Theorem}[section]

\newtheorem{corollary}[theorem]{Corollary}

\newtheorem{definition}{Definition}

\newtheorem{lemma}[theorem]{Lemma}

\newtheorem{proposition}[theorem]{Proposition}

\theoremstyle{remark}
\newtheorem{remark}[theorem]{Remark}

\numberwithin{equation}{section}

\newcommand{\pfrak}{\mathfrak{p}}

\newcommand{\afrak}{\mathfrak{a}}

\newcommand{\F}{\mathbb{F}}

\newcommand{\N}{\mathbb{N}}

\newcommand{\G}{\mathbb{G}}

\newcommand{\ord}{\mathrm{ord}}
\newcommand{\End}{\mathrm{End}}

\newcommand{\Nr}{\mathrm{Nr}}

\newcommand{\Zcali}{\mathcal{Z}}

\newcommand{\hhat}{\hat{h}}

\newcommand{\phitx}{\phi_T(x)}

\newcommand{\phitnx}{\phi_{T^n}(x)}
\newcommand{\logmax}{\mathrm{log}_q \ \mathrm{max}}
\newcommand{\phitor}{\phi_{\mathrm{Tor}}}

\newcommand{\ds}{\displaystyle}
\newcommand{\Afrak}{\mathfrak{A}}
\newcommand{\Bfrak}{\mathfrak{B}}
\newcommand{\Fqt}{\F_q[T]}
\newcommand{\FFqt}{\F_q(T)}

  \DeclareFontFamily{U}{wncy}{}
    \DeclareFontShape{U}{wncy}{m}{n}{<->wncyr10}{}
    \DeclareSymbolFont{mcy}{U}{wncy}{m}{n}
    \DeclareMathSymbol{\Sha}{\mathord}{mcy}{"58}

\begin{document}
\title{Explicit Zsigmondy bounds for families of Drinfeld modules of rank 2}

\author{Mat\'ias Alvarado}
\address{ Departamento de Matem\'aticas,
Pontificia Universidad Cat\'olica de Chile.
Facultad de Matem\'aticas,
4860 Av.\ Vicu\~na Mackenna,
Macul, RM, Chile}
\email[M. Alvarado]{mnalvarado1@uc.cl }%

\thanks{Supported by ANID Doctorado Nacional 21200910}

\date{\today}
\subjclass[2010]{Primary 11G09; Secondary 11S82, 37P15, } %
\keywords{Drinfeld modules, Zsigmondy bound, height}

%%%%%%%%%
%%%%%%%%%
%%%%%%%%%
%%%%%%%%%
\begin{abstract} We give explicit bounds for Zsigmondy sets of certain families of Drinfeld modules of rank 2. The primary strategy is to bound the local heights associated to Drinfeld modules and then relate canonical to classical heights.
\end{abstract}

%%%%%%%%%
%%%%%%%%%
%%%%%%%%%
%%%%%%%%%

\maketitle
\setcounter{tocdepth}{1}

\section{Introduction}

Let $p$ be a prime and $q$ be a power of $p$. Given a Drinfeld module $\phi$ over $\FFqt$ and an element $x\in \FFqt$, we study arithmetic aspects of the orbit of $x$ via the rational function $\phi_T.$ The element $\phitnx$ is written as $\phitnx=\Afrak_n \Bfrak_n^{-1}$, with $\Afrak_n$ and $\Bfrak_n$ coprime polynomials. A primitive divisor of $\Afrak_n$ is a prime element $\pfrak$ in $\Fqt$ such that $\pfrak|\Afrak_n$ and $\pfrak \nmid \Afrak_i$ for every $0\leq i <n$. Then the Zsigmondy set associated to $\phi$ and $x$ is defined by
$$\Zcali(\phi,x)=\{n \geq 1 : \Afrak_n \text { does not have a primitive divisor }\}.$$
The finiteness of the Zsigmondy set attached to a Drinfeld module $\phi$ and a non-torsion point $x$ in $\FFqt$ has been proved in \cite{zhaoji} by Ji and Zhao. Unfortunately, the proof does not give a bound on the size nor the maximum of $\Zcali(\phi,x)$. In fact, it is a challenging task to find these in general. One of the main difficulties is understanding the behavior of certain local heights. 
In this article, we study of families of Drinfeld modules of rank 2 for which we can bound  $\max  \Zcali(\phi,x)$, and for these families, our bound depends only on the degree and the number of prime divisors of the coefficients.

In the previous setting, sequence $\Afrak_n$ is analogous to elliptic divisibility sequences (EDS). The EDS are essential in arithmetic dynamics and play an important role in problems of decidability, e.g., see \cite{poonenh10}

If $K$ is a function field, then $M_K$ will be the set of places of $K$. Henceforth, $\log_q$ always means the logarithm to the base $q.$ Our main result is the following:

\begin{theorem}\label{mainthm}
Let $\phi \colon \F_q[T] \to \End_{\F_{q}(T)}(\G_a)$ be a Drinfeld module of rank $2$ given by
$$\phi_T(x)=Tx+gx^q+\Delta x^{q^2},$$
where $g,\Delta \in \F_q[T]$. Let $N$ be the least common multiple of the degrees of irreducible divisors of $\Delta$. Let us suppose that $\phi_T(1)$ is coprime to $T\Delta$ in $\Fqt.$ Let $S$ be the finite set of places $\{v \in M_{\FFqt}:v\, |\, T \Delta\}\cup \{v_\infty\}.$ 
If $x \in \Fqt$ is not a torsion point of $\phi$, then 
$$\max \Zcali(\phi,x) \leq \left(6(q^{2N}-1)N|S|+4\right)+\dfrac{1}{2}\log_q \left(2+2\deg g+2\deg \Delta  \right).$$ 
\end{theorem}

We first prove this under the assumption that $N=1$ (see Section \ref{proofmainthm}) and then deduce the general case.
\noindent In what follows, $S$ will denote a finite set of places as in Theorem \ref{mainthm}. To ease the notation, let $\Theta=q^r,$ where $r$ is the rank of the corresponding Drinfeld module when we discuss higher rank Drinfeld modules.

The structure of the paper is as follows. In Section 2, we study heights in function fields and establish bounds for the difference between the classical height and the canonical height associated to a Drinfeld module. In Section 3, we introduce the main idea of the proof, inspired by the results in \cite{zhaoji}. In Section 4, we study the behavior of local heights. Finally, in Section 5, we establish Theorem \ref{mainthm} for the case $N=1$, and then, we use this particular case to conclude the result for general $N$.

\section{Heights}\label{notation}

\subsection{Heights and Drinfeld modules}
In this section, we recall two notions of heights of points $x\in \overline{\F_q(T)}$. First, we define the classical height.
\begin{definition}
Let $x$ be an algebraic element over $\F_q(T)$, and let $K/\F_q(T)$ be a finite extension such that $x\in K$. We define the \emph{height} of $x$, denoted by $h(x)$ as
$$h(x)=\dfrac{1}{[K:\FFqt]}\sum_{v \in M_K} \logmax\{|x|_v,1\}.$$
\end{definition}

\begin{remark}
If $x=a(T)/b(T) \in \FFqt$,  such that $\gcd(a(T),b(T))=1,$ then $h(x)$ can be expressed as
$$h(x)=\max \left\{\deg a(T),\deg b(T)\right\}.$$
\end{remark}

\noindent In \cite{denis}, Denis defined a height corresponding to a Drinfeld module called canonical height. This height is the analog of the N\'eron-Tate height on elliptic curves.

\begin{definition}[Denis, \cite{denis}]
Let $\phi$ be an $\Fqt$-Drinfeld module of rank $r$ defined over an algebraic extension K of $\FFqt$. The \emph{global canonical height} attached to $\phi$ is defined as follows: For $x\in K$,
$$\hhat(x)=\lim_{n\to\infty}\dfrac{h(\phi_{T^n}(x))}{\Theta^{n}}.$$
\end{definition}

\begin{remark}
Following Th\'eor\`eme 1 in \cite{denis}, if $a$ is any non-constant element in $\Fqt$, then we can calculate $\hhat(x)$ as follow
$$\hhat(x)=\lim_{n\to \infty} \dfrac{h(\phi_{a^n}(x))}{(\deg \phi_{a}(x))^n}.$$
\end{remark}

\noindent 
Now, we give a list of height properties that we will use to prove the main theorem.
If  $\phi$ is a Drinfeld module defined over a field $k$, given by $\phi_T(x)=Tx+g_1x^q+\cdots g_{r}x^{q^r}$ and $\gamma \in \overline{\FFqt}^{\times}$, then we can obtain another Drinfeld module by conjugating the Drinfeld module $\phi$ by multiplication by $\gamma,$ and it is denoted by $\phi^{(\gamma)}$. If $a\in \F_q[T]$, then $\phi^{(\gamma)}_a(x)=Tx+g_1\gamma^{q-1}x^{q}+\cdots +g_r \gamma^{q^r-1}x^{q^r}$. This new Drinfeld module $\phi^{(\gamma)}$ is defined over $k(\gamma^{1/(q^r-1)})$.
If $\afrak$ is an ideal of $\Fqt$, then $\Nr(\afrak)=\#(\Fqt/\afrak)$. The norm of the infinite place is $\Nr(v_\infty)=q$.
For a finite set of primes $S$ in $\Fqt$, we define $\ds \Nr^S(\afrak)=\prod_{\pfrak \notin S}\Nr(\pfrak)^{\ord_\pfrak(\afrak)}.$ For $a(T)\in \F_q[T]$, we denote by $\Nr(a(T))$ the norm of the ideal generated by $a(T)$. Similarly, $\Nr^S(a(T))$ represent the norm $\Nr^S$ of the ideal generated by $a(T).$ 

\begin{proposition}\label{properties} Let $\phi$ be an $\F_q[T]$-Drinfeld module of rank $r$ defined over $\overline{\FFqt}$, $\gamma  \in \overline{\FFqt}^\times$, and $x \in \FFqt$, then
\hfill
\begin{enumerate}
\item[(i)]$\hhat_{\phi}(x)=\hhat_{\phi^{(\gamma)}}(\gamma^{-1}x)$.
\item[(ii)] $\hhat(\phitnx)=\Theta^n\hhat(x)$.
\end{enumerate}
\end{proposition} 
 
 \begin{proof}
Item (i) follows by Proposition 2 in \cite{poonen}. To prove (ii), let $x$ be an element in $\FFqt$, then
\begin{align*}
\hhat(\phi_{T^n}(x))&=\lim _{m \to \infty} \dfrac{h(\phi_{T^m}(\phi_{T^n}(x)))}{\Theta^m}\\
&=\lim_{m \to \infty} \Theta^n \dfrac{h(\phi_{T^{n+m}(x)})}{\Theta^{n+m}}\\
&=\Theta^n\hhat(x),
\end{align*}
which gives the desired result.
\end{proof}

\begin{proposition}\label{proph} Let $x=a(T)/b(T)$ be an element in $\FFqt$, with $a(T)$ and $b(T)$ coprime polynomials. Let $S$ be a finite set of places containing $v_\infty$. Then
$$h(x)=\log_q \Nr^S(a(T))+\sum_{v\in S}\logmax\{1,|x|^{-1}_v\}.$$
\end{proposition}
\begin{proof}
Notice $h(x)=h(x^{-1})$. Thus, we only need to prove $\Nr^S(a(T))=\prod_{v\notin S} \max\{1,|x^{-1}|_v\}$. If $v\notin S$ and $v(a(T))>0$, then $$\max\{1,|x^{-1}|_v\}=\max\{1,q^{v(x)\deg(v)}\}=q^{deg(v)v(x)}=(q^{deg(v)})^{v(a(T))}=\Nr_v(a(T)),$$ where $\Nr_v(a(T))=\Nr(\pfrak_v)^{v(a(T))}$, and $\pfrak_v$ the prime ideal associated to the place $v$.  When $v(a(T))=0$, both $\Nr_v(a(T))$ and $ \max\{1,|x^{-1}|_v\}$ are equal to $1$. We conclude that $\Nr_v(a(T))=\max\{1,|x^{-1}|_v\}$ for all $v \notin S.$
\end{proof}

\noindent The following corollary follows immediately from Proposition \ref{proph}.

\begin{corollary}\label{coroheight}
If $x\in \F_q(T)$, and $\phi_{T^n}(x)=\Afrak_n/\Bfrak_n$, then $$\ds h(\phitnx)=\log_q \Nr^S(\Afrak_n)+\sum_{v \in S}\logmax \{1,|\phitnx^{-1}|_v\}.$$
\end{corollary}

\subsection{Difference between $h$ and $\hhat$}\label{difference}

\noindent  To establish bounds for the maximum of Zsigmondy sets, we base our approach on understanding the difference between the classical height $h$ and the canonical height $\hhat$. We estimate this difference and obtain bounds independent of the element $x$ in $\Fqt$. The bounds depend on the coefficients of the Drinfeld module $\phi.$ Let $\phi$ be a rank $r$ $\F_q[T]$-Drinfeld module defined over $\F_q(T)$ such that $\phi_T(x)=Tx+c_1x^q+...+c_rx^{\Theta}$, where $c_i\in \F_q[T]$, and $c_r\neq 0$. Then, we define two constants associated to $\phi$,
\begin{align*}
M_\phi&=\dfrac{q}{(\Theta-1)(q-1)}\max\{1-\deg c_r,...,\deg c_{r-1}-\deg c_r,0\},\\
M'_\phi&=\dfrac{1}{(\Theta-1)} \max\{1,\deg c_1,...,\deg c_r\}.
\end{align*}

\noindent By Proposition $6$ in \cite{poonen}, we can decompose the heights $h$ and $\hhat$ as the sum of the local heights $h_v(x)$ and $\hhat_v(x)$, respectively, where

$$h_v(x)=\logmax\{1,|x|_v\}, \text{ \ \ and \ \ } \hhat_v(x)=\lim_{n \to \infty} \dfrac{\logmax\{1,|\phi_{T^n}(x)|_v\}}{\Theta^n}.$$

\noindent Now, we state the results concerning the bounds for the difference of heights.

\begin{lemma}\label{diffh1}
Let $\phi$ be a rank $r$ $\F_q[T]$-Drinfeld module defined over $\FFqt$ such that $\phitx=Tx+c_1x^q+\cdots + c_rx^{q^r}$, with $c_i \in \F_q[T]$ for $0\leq i <r$ and $c_r \neq 0$. Then, for all $x\in \F_q[T],$ we have 
$$h(x)-\hhat(x) \leq M_\phi.$$

\end{lemma}

\begin{proof}
The first step in the proof is to bound the expression $\logmax \{|x|_v^{\Theta},1\}-\logmax\{|\phi_T(x)|_v,1\}$ for each place $v$ in $M_{\F_q(T)}$. From now on, we will denote the previous difference by  $\kappa_v$. Since $\phi_T$ has coefficients in $\F_q[T]$, and $x\in \F_q[T]$, both $|x|_v$ and $|\phi_T(x)|_v$ are bounded by $1$ for all finite places $v.$ Therefore, for these places, we have $\logmax \{|x|_v^{\Theta},1\}-\logmax\{|\phi_T(x)|_v,1\}=0.$ Thus, we only need to bound $\logmax \{|x|_v^{\Theta},1\}-\logmax\{|\phi_T(x)|_v,1\}$ when $v$ is the infinite place $v_\infty.$

\noindent To achieve this, first, we suppose that $|c_rx^{\Theta}|_{v_\infty} > |c_ix^{q^i}|_{v_\infty}$ for all $i \in \{0,...,r-1\}$, where we use the convention $c_0=T.$ Then,
$$\kappa_{v_\infty}=\log_q|x|_{v_\infty}^{\Theta}-\log_q|x|_{v_\infty}^{\Theta}-\log_q|c_r|_{v_\infty}=-\log_q|c_r|_{v_\infty}=-\deg c_r.$$

\noindent Now, if $|x|_{v_\infty} \leq \left|\dfrac{c_i}{c_r} \right|_{v_\infty}^{\frac{1}{\Theta-q^i}}$ for some $i\in \{0,...,r-1\},$ we have

$$|x|_{v_\infty}^{\Theta}\leq \max \left\{  \left|\dfrac{T}{c_r}\right|_{v_\infty}^{\frac{\Theta}{\Theta-1}}, \left|\dfrac{c_1}{c_r}\right|_{v_\infty}^{\frac{\Theta}{\Theta-q}},..., \left|\dfrac{c_{r-1}}{c_r}\right|_{v_\infty}^{\frac{\Theta}{\Theta-q^{r-1}}} \right\}.$$
Therefore

\begin{align*}
\kappa_{v_\infty} &\leq \logmax \left\{  \left|\dfrac{T}{c_r}\right|_{v_\infty}^{\frac{\Theta}{\Theta-1}}, \left|\dfrac{c_1}{c_r}\right|_{v_\infty}^{\frac{\Theta}{\Theta-q}},..., \left|\dfrac{c_{r-1}}{c_r}\right|_{v_\infty}^{\frac{\Theta}{\Theta-q^{r-1}}} \right\}\\
&\leq \dfrac{q}{(q-1)}\max\{1-\deg c_r,...,\deg c_{r-1}-\deg c_r,\}.
\end{align*}
\noindent This way, we conclude that
$$\kappa_{v_\infty} \leq \dfrac{q}{(q-1)}\max\{1-\deg c_r,...,\deg c_{r-1}-\deg c_r,0\} .$$

\noindent Let us denote by $m_\phi$ the expression $\dfrac{q}{(q-1)}\max\{1-\deg c_r,...,\deg c_{r-1}-\deg c_r,0\}.$ 

\noindent The second step is to bound 
$\logmax\{|x|_v,1\}-\dfrac{1}{\Theta^n} \logmax \{|\phi_{T^n}(x)|_v,1\}.$ Following the same strategy as in the first step, we obtain that $\phi_{T^n}(x)\in \F_q[T],$ since $x\in \F_q[T]$. Then, $\logmax\{|x|_v,1\}-\dfrac{1}{\Theta^n} \logmax \{|\phi_{T^n}(x)|_v,1\}=0$ for all finite places. On the other hand, if $v=v_\infty,$ we have
\begin{align*}
\logmax\{|x|_{v_\infty},1\}-&\dfrac{1}{\Theta^n} \logmax \{|\phi_{T^n}(x)|_{v_\infty},1\}\\
&= \dfrac{1}{\Theta^n} \left( \logmax\{|x|_{v_\infty}^{\Theta^n},1\}-\logmax\{|\phi_{T^n}(x)|_{v_\infty},1 \} \right) \\
&= \sum_{k=0}^{n-1}\dfrac{1}{\Theta^{k+1}}\left( \logmax\{|\phi_{T^k}(x)|_{v_\infty}^{\Theta},1\} - \logmax \{|\phi_{T^{k+1}}(x)|_{v_\infty},1\}\right) \\
&\leq \sum_{k=1}^n \dfrac{1}{\Theta^k} m_\phi.
\end{align*}

\noindent Therefore, taking the limit $n \to \infty$, we conclude

$$h_{v_\infty}(x)-\hhat_{v_\infty}(x)\leq m_\phi \sum_{k=1}^{\infty}\dfrac{1}{\Theta^k}=\dfrac{m_\phi}{\Theta-1}=M_\phi.$$

\noindent Finally, we note $\ds h(x)-\hhat(x)=h_{v_\infty}(x)-\hhat_{v_\infty}(x) \leq M_\phi.$
\end{proof}

\begin{lemma}\label{lemmam'}
Let $\phi$ be a rank $r$ $\Fqt$-Drinfeld module defined over $\FFqt$ such that $\phitx=Tx+c_1x^q+\cdots + c_rx^{\Theta}$ with $c_i \in \Fqt$ for $0 \leq i <r$ and $c_r\neq 0$. Then for all $x \in \F_q[T]$, we have 
$$\hhat(x)-h(x) \leq M'_\phi.$$
\end{lemma}

\begin{proof}

As in the proof of Lemma \ref{diffh1}, the first step is to estimate the difference between $\logmax\{|\phi_T(x)|_v,1\}$ and $\logmax\{|x|^{\Theta},1\}$. Again, we only study this difference for the infinite place since $x\in \F_q[T]$. If we suppose $|c_rx^{\Theta}|_{v_\infty}>|c_ix^{q^i}|_{v_\infty}$ for all $0\leq i <r,$ then $\logmax \{|\phi_T(x)|_{v_\infty},1\}-\logmax\{|x|_{v_\infty}^\Theta\,1\}=\log_q|c_r|_{v_\infty}=\deg c_r$. On the other hand, we suppose  $|x|_{v_\infty} \leq \left| \dfrac{c_i}{c_r} \right|_{v_\infty}^{\frac{1}{\Theta-q^i}}$ for some $i<r,$ and that $|c_ix^{q^i}|$ is the maximum between $|Tx|,|c_1x^q|,...,|c_rx^{\Theta}|$. Therefore,

\begin{align*}
-\kappa_{v_\infty}
&\leq \logmax\{|c_ix^{q^i}|_{v_\infty},1\}-\log_q |x|_{v_\infty}^\Theta \\
&= \max \left\{\log_q|c_i|_{v_\infty}+\log_q|x|_{v_\infty}^{q^i-\Theta},\log_q|x|_{v_\infty}^{-\Theta}\right\}\\
&\leq\deg c_i \\
&\leq \max\{1,\deg c_1,...,\deg c_r\} =: m'_\phi.
%&\leq \max \left\{ \log|c_i|_{v_\infty}+\log\left|\frac{c_i}{c_r} \right|_{v_\infty}^{-1}, \log \left|\frac{c_i}{c_r}\right|_{v_\infty}^{\frac{-\Theta}{\Theta-q^i}} \right\}\\
%&= \max \left\{ \log |c_r|_{v_\infty},\log\left| \frac{c_r}{c_i} \right|^{\frac{\Theta}{\Theta-q^i}}_{v_\infty} \right\} \\
%&\leq \max \left\{ \log |c_r|_{v_\infty},\log\left| \frac{c_r}{T} \right|^{\frac{\Theta}{\Theta-1}}_{v_\infty},...,\log\left| \frac{c_r}{c_{r-1}} \right|^{\frac{\Theta}{\Theta-q^{r-1}}}_{v_\infty} \right\} \\
%&\leq \dfrac{q}{q-1}\deg c_r
\end{align*}
The second equality holds because $|x|_{v_\infty}>1$. As in the previous estimation, we obtain
$$\hhat(x)-h(x)=\hhat_{v_\infty}(x)-h_{v_\infty}(x )\leq \dfrac{1}{\Theta-1}m'_\phi=M_\phi'.$$
\end{proof}

\subsection{Bounds in rank 2}\label{boundsrank2}

Although the bounds from the previous section for the difference between the heights hold for Drinfeld modules of arbitrary rank, we use them in the particular case of rank 2. In this subsection, we give the explicit bounds in terms of the coefficients $g$ and $\Delta$ to establish a bound for the Zsigmondy set $\Zcali(\phi,x)$ as a function of $g$ and $\Delta$.

\noindent Following the notation of Subsection \ref{difference}, for a Drinfeld module $\phi$ given by $\phi_T(x)=Tx+gx^{q}+\Delta x^{q^2}$, we have 

\begin{align*}
M_\phi &= \dfrac{q}{(q^2-1)(q-1)}\max\{1-\deg \Delta, \deg g-\deg\Delta,0\} \text{ and }\\
M_\phi' &= \dfrac{1}{(q^2-1)}\max\{1,\deg g,\deg \Delta\}.
\end{align*}

\noindent Then, we get the following bounds for the difference between $h$ and $\hhat$
\begin{align*}
h(x)-\hhat(x) &\leq \dfrac{1}{(q-1)^2}(1+\deg g) \text{ and } \\
\hhat(x)-h(x) &\leq \dfrac{1}{(q-1)^2}(\deg \Delta+\deg g+1).
\end{align*}

%%%%%%%%%%%%%%%%%%%%%%%%%%%%%%%%%%%%%%%%%%%%%%%
\section{Previous results and strategy}

In the proof of Theorem 4.6 in \cite{zhaoji}, the authors obtain the following inequality
\begin{equation}\label{n0}
\log_q \left( \dfrac{\Nr^S(\Afrak_n)}{\Nr^S(\Afrak_{n-1})} \right) \geq \left( \Theta-\dfrac{3}{2} \right)\Theta^{n-1}\hhat(x)-M_\phi.
\end{equation}
This inequality holds for all natural numbers greater than $n_0$. The existence of such $n_0$ is proved in Lemma 4.5 (1) in \cite{zhaoji}.

Our strategy to prove the main theorem in this article is to find such $n_0$ and find explicitly bounds on $n$ such that 

$$\left( \Theta-\dfrac{3}{2} \right)\Theta^{n-1}\hhat(x)-M_\phi >0, \text{ for all } n\geq n_0.$$

%The strategy followed in \cite{zhaoji} to prove the finiteness of the Zsigmondy sets is to give a lower bound for $\log \left(\frac{\Nr^S(\Afrak_n)}{\Nr^S(\Afrak_{n-1})}\right)$. Firstly, they prove that there exists $n_0 \in \N$ such that 

%\begin{equation}\label{n0}
%\left(1-\dfrac{1}{2\Theta}\right)\Theta^n \hhat(x) \leq \log \Nr^S(\Afrak_n), \text{ for all } n \geq n_0.
%\end{equation}

%and then, they proved that there exists a constant $M$ such that 
%\begin{equation}\label{eq2}
%\log \left( \dfrac{\Nr^S(\Afrak_n)}{\Nr^S(\Afrak_{n-1})} \right) \geq \left( \Theta-\dfrac{3}{2} \right)\Theta^{n-1}\hhat(x)-M, \text{ for all } n \geq n_0. 
%\end{equation}
%In our case, we can take the constant $M$ equal to $M_\phi$.
 
%Inspired by the bounds for $\log \Nr(\Afrak_n)$ and $\log \Nr^S(\Afrak_n)$ given in Lemma 4.5 in %\cite{zhaoji}, we will find an explicit value of $n_0$ such that inequality \ref{n0} holds.

\begin{definition}
For a place $v$ and a sequence of elements $\{\Afrak_n\}_{n\geq 1}$ in $\Fqt$, the rank of apparition at $v$ is defined by the integer
$$r_v=\min \{n\geq 1:\ord_v(\Afrak_n)>0\}.$$
If no such $n$ exists, $r_v$ is set to be $\infty.$
\end{definition}

\begin{lemma}[\cite{zhaoji}]\label{aparition} For every finite place $v$ such that $v(T)=0$, and a sequence of polynomials $\{\Afrak_n\}$ derived from $\{\phitnx\}_{n\geq 1}$, if $\ell \leq r_v$, we have $\ord_v(\Afrak_{\ell-1})=0$; otherwise, $\ord_v(\Afrak_\ell)=\ord_v(\Afrak_{\ell-1})$.
\end{lemma}

\begin{proof}
See Lemma 4.3 in \cite{zhaoji}.
\end{proof}

Having said this, we need to find such $n_0$, and in this way, thanks to equation \ref{n0}, we will be able to show such that $\log_q \left( \Nr^S(\Afrak_n)/\Nr^S(\Afrak_{n-1}) \right) > 0$ for all $n\geq n_0$.

%%%%%%%%%%%%%%%%%%%%%%%%%%%%%%%%%%%%%%%%%%

\section{Local heights}
\subsection{Estimates}

In order to understand the behavior of $h(\phitnx)$, by Proposition \ref{proph}, we have to find an upper bound for $\sum_{v\in S} \logmax \left\{|\phi_{T^n}(x)^{-1}|_v,1 \right\}.$
We note that
$$\sum_{v\in S} \logmax \left\{|\phi_{T^n}(x)^{-1}|_v,1 \right\}=\sum_{v\in S} \deg(v) \cdot \max \left\{v(\phi_{T^n}(x)),0 \right\}.$$
Consequently, it is enough to study $\max \left\{v(\phi_{T^n}(x)),0 \right\}$ for $v\in S.$

\begin{lemma}\label{localheight}
Let $\phi \colon \F_q[T] \to \End_{\F_q(T)}(\G_a)$ be a Drinfeld module of rank $2$ given by
$$\phi_T(x)=Tx+gx^q+\Delta x^{q^2}$$
such that $\Delta \in \Fqt$ split completely over $\F_q$, $g\in \F_q[T]$  and $\phi_T(1)$ is coprime to $T\Delta$ in $\Fqt.$ Then for all $x \in \F_q[T]$, we have

$$\sum_{v \in S} \logmax \left\{|\phi_{T^n}(x)^{-1}|_v,1\right\}=0.$$

\end{lemma}

\begin{proof}
If {$v=v_\infty$}, then $v(\phi_{T^n}(x))\leq 0$ because $x \in \Fqt$, so $\deg(v_\infty) \cdot \max\{v(\phi_{T^n}(x)),0\}=0.$

\noindent On the other hand, if $v(T\Delta)>0$, by euclidean division, we can express $x$ as
\begin{equation}\label{eqeuclidean}
x=P(T)y+\gamma,
\end{equation}
where $\gamma\in \F_q$ and $P(T)$ is the irreducible polynomial corresponding to the place $v$. Evaluating $\phi_T$ in equation \ref{eqeuclidean}, 
\begin{align*}
\phi_T(x)
&=\phi_T(P(T)y)+\phi_T(\gamma) \\
&=P(T)\left( Tx+P(T)^{q-1}gx^q+P(T)^{q^2-1}\Delta x^{q^2-1} \right)+\gamma(T+g+\Delta).
\end{align*}

\noindent We note that $\gamma(T+g+\Delta)=\phi_T(1)$. Hence $\phi_T(1)$ is coprime to $T\Delta$ we have $v(\gamma(T+g+\Delta))=0$, and then we conclude $v(\phi_T(x))=0.$ Finally, we have

$$\sum_{v \in S} \logmax \left\{|\phi_{T^n}(x)^{-1}|_v,1\right\}=0.$$

\end{proof}
%We also can get a lower bound.
%\subsection{Other estimations}
%By lemma \cite{lemmazhaoji} (2), given $\epsilon >0$,
%$$\hhat(x)-\dfrac{\log \NN^S(A_n)}{\Theta^n}\leq \epsilon \hhat(x)$$ 
%for $n\geq n_0$.  

%In order to find such $n_0$, we first bound left-hand side in this inequality for rank $2$ Drinfeld modules over $\FFqt.$

\begin{lemma}\label{lemmahhat}
Let $\phi$ be a rank $2$ $\F_q[T]$-Drinfeld module defined over $\FFqt.$ Then, for all $x\in \F_q[T]$

$$\hhat(x)-\dfrac{\log _q \Nr^S(\Afrak_n)}{\Theta^n} \leq \dfrac{M'_\phi}{\Theta^n} +\dfrac{1}{\Theta^n}\sum_{v \in S} \logmax\{|\phitnx^{-1}|_v,1\}. $$
\end{lemma}

\begin{proof}
We add and subtract $\dfrac{h(\phitnx)}{\Theta^n}$ to the expression $\hhat(x)-\dfrac{\log_q \Nr^S(\Afrak_n)}{\Theta^n}$, then

\begin{equation}\label{eq1}
\hhat(x)-\dfrac{\log_q \Nr^S(\Afrak_n)}{\Theta^n}=\hhat(x)-\dfrac{h(\phitnx)}{\Theta^n}+\left( \dfrac{h(\phitnx)}{\Theta^n}-\dfrac{\log_q \Nr^S(\Afrak_n)}{\Theta^n} \right).
\end{equation}

\noindent By Proposition \ref{proph} (iii), the expression \ref{eq1} is equal to

\begin{equation}\label{eq222}
\left( \hhat(x)-\dfrac{h(\phitnx)}{\Theta^n} \right)+\dfrac{1}{\Theta^n}\sum_{v \in S} \logmax\{|\phitnx^{-1}|_v,1\}.
\end{equation}

Adding and subtracting $\dfrac{1}{\Theta^n}\hhat(\phitnx)$ we have that expression \ref{eq222} is equal to

\begin{equation}\label{eq3}
\dfrac{1}{\Theta^n} \left( \left(\Theta^n \hhat(x)-\hhat(\phitnx)\right) + \left(\hhat(\phitnx)-h(\phitnx) \right) \right) +\dfrac{1}{\Theta^n}\sum_{v \in S} \logmax\{|\phitnx^{-1}|_v,1\}.
\end{equation}

Now, by Proposition \ref{properties} (ii), we have that the expression \ref{eq3} is equal to

$$\dfrac{1}{\Theta^n} \left(\hhat(\phitnx)-h(\phitnx) \right)+ \dfrac{1}{\Theta^n}\sum_{v \in S} \logmax\{|\phitnx^{-1}|_v,1\},$$

and by Lemma \ref{lemmam'}, we conclude

$$\hhat(x)-\dfrac{\log_q \Nr^S(\Afrak_n)}{\Theta^n} \leq  \dfrac{M'_\phi}{\Theta^n} +\dfrac{1}{\Theta^n}\sum_{v \in S} \logmax\{|\phitnx^{-1}|_v,1\}.$$
\end{proof}

\begin{remark}
Later, we will use Lemma \ref{localheight} to control the term $$\frac{1}{\Theta^n}\sum_{v \in S} \logmax\{|\phitnx^{-1}|_v,1\}.$$
\end{remark}

\subsection{Bound for $\hhat(x)$}

In \cite{ghiocaheight}, Ghioca studies an analogous of Lehmer's conjecture for canonical height coming from a Drinfeld module. The following theorem is a particular case of Theorem 4.5 in \cite{ghiocaheight}.

\begin{theorem}[Ghioca \cite{ghiocaheight}]\label{boundheight}
Let $\phi \colon \F_q[T] \to \End_{\FFqt}(\G_a)$ be a Drinfeld module of rank $2$ such that $\phitx=Tx+gx^q+\Delta x^{q^2}$. If $\alpha \notin \phitor$, then
$$\hhat_\phi(\alpha)> q^{-6-12(q^2-1)|S|}.$$
\end{theorem}

\begin{proof}
Let $\gamma \in \overline{\FFqt}$, such that $\gamma^{q^2-1}=\Delta.$ Then $\phi_T^{(\gamma)}(x)$ is a monic polynomial in $x$ with coefficients in $\FFqt(\gamma)$. By \cite{ghiocalocalehmer} and Lemma \ref{properties}(i), we have $\hhat_\phi(x)  = \hhat_{\phi^{(\gamma)}}(\gamma^{-1}x).$ So, we need to bound the height $ \hhat_{\phi^{(\gamma)}}$. As we have said, $\phi^{(\gamma)}_T(x)$ is a monic polynomial, then we can use the bounds obtained in Theorem 4.4 in \cite{ghiocaheight}. We have to know the size of places in $\FFqt(\gamma)$ such that $\phi^{(\gamma)}$ has bad reduction. We observe that a place in $\FFqt(\gamma)$ of bad reduction of $\phi^{(\gamma)}$ is always over a place of $S$. We denote by $S_{\phi^{(\gamma)}}$ the set of bad places of $\phi^{(\gamma)}.$ Since $[\F_q(T)(\gamma^{-1}x):\F_q(T)] \leq (q^2-1)$, we obtain $|S_{\phi^{(\gamma)}}| \leq (q^2-1)|S|$.

Then, applying the bound in Theorem 4.4 from \cite{ghiocaheight}, we obtain

$$\hhat_\phi(x)  = \hhat_{\phi^{(\gamma)}}(\gamma^{-1}x)
> \dfrac{q^{-4-12\left|S_{\phi^{(\gamma)}}\right|}}{[\FFqt(\gamma^{-1}x):\FFqt]}.$$

In this way, we conclude
$$\hhat_\phi(x)> q^{-6-12(q^2-1)|S|}.$$

\end{proof}

%%%%%%%%%%%%%%%%%%%%%%%%%%%%%%%%%%%%%%%%
\section{Proof of the main theorem}\label{proofmainthm}

\begin{lemma}\label{keylemma}
Let $\phi \colon \F_q[T] \to \End_{\FFqt}(\G_a)$ be a Drinfeld module of rank $2$ given by
$$\phi_T(x)=Tx+gx^q+\Delta x^{q^2},$$
such that $\Delta \in \Fqt$ split completely over $\F_q$, $g\in \F_q[T]$  and $\phi_T(1)$ is coprime to $T\Delta$ in $\Fqt.$ If $x \in \Fqt$  is not a torsion point of $\phi$, then 
$$\max\Zcali(\phi,x) \leq \left(6(q^2-1)|S|+4\right)+\dfrac{1}{2}\log_q \left(2+2\deg g+2\deg \Delta  \right).$$ 
\end{lemma}

\begin{proof}
First, we need to find the minimum $n_0$ such that equation \ref{n0} holds. Using Lemma \ref{lemmahhat}, we see that it is enough take $n_0$ as the minimum $n \in \N$ satisfying the following inequality

$$\dfrac{M_\phi'}{\Theta^n} \leq \dfrac{1}{2\Theta} \hhat(x).$$ 

\noindent From this expression, we obtain that inequality \ref{n0} holds if

\begin{equation}\label{eq22}
n \geq \dfrac{1}{2}\left( \log_q (2M'_\phi)-\log_q \hhat(x) \right)+1.
\end{equation}

\noindent On the other hand, to ensure that $\log_q \left(\frac{\Nr^S(\Afrak_n)}{\Nr^S(\Afrak_{n-1})}\right)>0$, by equation \ref{n0}, we need

$$\left(\Theta-\dfrac{3}{2} \right)\Theta^{n-1} \hhat(x) > M_\phi.$$
Applying logarithms and dividing by $\log_q \Theta$, we obtain 

$$n-1> \dfrac{1}{\log_q \Theta} \left(\log_q M_\phi - \log_q \hhat(x)-\log_q \left(\Theta-\frac{3}{2}\right) \right).$$
Recalling $\Theta=q^r$, where $r$ is the rank of the corresponding Drinfeld module (in our case, $r=2$), we have $\log_q(\Theta)=2$.
Then, it is enough that \begin{equation}\label{eq11}
n > \dfrac{1}{2} \left(\log_q M_\phi-\log_q \hhat(x)-\log_q \left(\Theta-\frac{3}{2}\right) \right)+1.
\end{equation}

\noindent Using the bounds obtained in subsections \ref{boundsrank2} and \ref{boundheight}: $$M_\phi \leq \dfrac{1+\deg g}{(q-1)^2}, M_\phi' \leq  \dfrac{2(\deg \Delta+\deg g+1)}{(q-1)^2}, \text{ and } -\log_q(\hhat(x)) \leq 12(q^2-1)|S|+6,$$

\noindent we can conclude that inequality \ref{eq11} holds if 
$$n\geq (6(q^2-1)|S|+4)+\dfrac{1}{2}  \log_q \left( \dfrac{2(\deg \Delta+\deg g +1)}{(q-1)^2}\right), $$

\noindent and that inequality \ref{eq22} holds if
$$n \geq (6(q^2-1)|S|+4)+\dfrac{1}{2}\log_q \left( \dfrac{1+\deg g}{(q-1)^2} \right).$$

\noindent Finally, we note that a sufficient condition on $n$ in order to \ref{eq11} and \ref{eq22} hold is  
$$n \geq (6(q^2-1)|S|+4)+\dfrac{1}{2}\log_q \left(2+2\deg g+2\deg \Delta  \right).$$

\noindent We conclude that $\max \Zcali(\phi,x) \leq (6(q^2-1)|S|+4)+\dfrac{1}{2}\log_q \left(2+2\deg g+2\deg \Delta  \right)$ under the hypothesis of Lemma \ref{keylemma}, which concludes the proof.

\end{proof}

\noindent Now, noting that Lemma \ref{keylemma} is the particular case of Theorem \ref{mainthm} when $N=1$, we use it to prove the main theorem.

\begin{proof}(of Theorem \ref{mainthm})
Let $\psi$ be an auxiliary Drinfeld module $\psi \colon \F_{q^N}[T] \to \End_{\F_{q^N}(T)}(\G_a)$ given by $\psi_T(x)=\phi_T(x)$. The advantage in this case is that $\Delta$ split completely over $\F_{q^N}$. 
Let us recall the definition of the sets $\Zcali(\phi,x)$ and $\Zcali(\psi,x),$

$$\Zcali(\phi,x)=\left\{n \in \N : \Afrak_n \text{ does not have primitive divisors in }  \F_q(T)\right\}.$$

$$\Zcali(\psi,x)=\left\{n \in \N : \Afrak_n \text{ does not have primitive divisors in }  \F_{q^N}(T) \right\}.$$

We can apply Lemma \ref{keylemma} and the fact that the set $S'$ of places in $\F_{q^N}(T)$  associated with the Drinfeld module $\psi$ has cardinality bounded by $N|S|$. Therefore,
$\Zcali(\psi,x)$ satisfies the bound of the theorem.
Now we only need to prove $\Zcali(\phi,x)=\Zcali(\psi,x)$ for all $x$ in $\F_q[T]$.

Let $n\in \N$ be an element in $\Zcali(\phi,x)$. Then, by definition of $\Zcali(\phi,x)$, $\Afrak_n$ does not have primitive divisors in $\F_q[T]$. Consequently, it is sufficient to prove $\Afrak_n$ does not have primitive divisors in $\F_{q^N}[T]$. If $v$ is a place in $\F_{q^N}(T)$ such that $v|\Afrak_n$, then there is  a place $w$ in $\F_q(T)$ below $v$. Since $\Afrak_n$ belongs to $\F_q[T]$, $w|\Afrak_n.$ But $n \in \Zcali(\phi,x)$, then there exists $m < n$ such that $w|\Afrak_m.$ Since $v$ is above $w$, $v|\Afrak_m.$ So, we conclude that $v$ is not a primitive divisor of $\Afrak_n$, therefore,
$$\Zcali(\phi,x) \subset \Zcali(\psi,x).$$

On the other hand, if $n\in \Zcali(\psi,x)$, then $\Afrak_n$ does not have primitive divisors in $\F_{q^N}[T]$. When $v|\Afrak$, there exists $m<n$ such that $v|\Afrak_m$, since $\Afrak_m \in \F_q[T]$. Then the place $w$ in $\F_q[T]$ below $v$, divide $\Afrak_m$. In this way, we conclude $n \in \Zcali(\phi,x)$, and therefore $\Zcali(\phi,x)=\Zcali(\psi,x).$

\end{proof}

\section*{Acknowledgement}
I would like to thank Hector Pasten for many suggestions and helpful remarks. 
Comments by Jerson Caro and Cristian Gonz\'alez-Riquelme on an earlier version of this manuscript are gratefully acknowledged.
I was supported by ANID Doctorado Nacional 21200910.

%%%%%%%%%%%%%%%%%%%%%%%%%%%%%%%%%%%%%%
%%%%%%%%%%%%%%%%%%%%%%%%%%%%%%%%%%%%%%

\section*{\textbf{Declarations}}

\textbf{Data availability} All data generated or analysed during this study are included in this published article and the cited references.

\textbf{Conflict of interest} On behalf of all authors, the corresponding author states that there is no conflict of interest.

\end{document}